\documentclass[11pt]{article}

\usepackage{graphicx, subfigure}
\usepackage[top=.75in,bottom=.75in,left=.75in,right=.75in]{geometry}
\usepackage{amsfonts, amsmath, amssymb, amsthm, constants,comment}
\usepackage{algorithm,algorithmic}
\usepackage[usenames]{color}

\definecolor{darkred}{RGB}{100,0,0}
\definecolor{darkgreen}{RGB}{0,100,0}
\definecolor{darkblue}{RGB}{0,0,150}
\definecolor{red}{RGB}{255,0,0}

\usepackage{hyperref}
\hypersetup{colorlinks=true, linkcolor=darkred, citecolor=darkgreen, urlcolor=darkblue}
\usepackage{url}

\newtheorem{thm}{Theorem}
\newtheorem{prp}{Proposition}
\newtheorem{lem}{Lemma}

\def\beq{\begin{equation}}
\def\eeq{\end{equation}}
\def\beqn{\begin{eqnarray*}}
\def\eeqn{\end{eqnarray*}}
\def\bitem{\begin{itemize}}
\def\eitem{\end{itemize}}
\def\benum{\begin{enumerate}}
\def\eenum{\end{enumerate}}
\def\bmult{\begin{multline*}}
\def\emult{\end{multline*}}
\def\bcenter{\begin{center}}
\def\ecenter{\end{center}}

\newcommand{\thmref}[1]{Theorem~\ref{thm:#1}}

\newcommand{\secref}[1]{Section~\ref{sec:#1}}
\newcommand{\figref}[1]{Figure~\ref{fig:#1}}
\newcommand{\algref}[1]{Algorithm~\ref{alg:#1}}

\def\cN{\mathcal{N}}

\def\bA{\mathbf{A}}
\def\bB{\mathbf{B}}

\def\bU{\mathbf{U}}

\def\ba{\mathbf{a}}

\def\be{\mathbf{e}}

\def\bp{\mathbf{p}}

\def\bu{\mathbf{u}}

\def\bx{\mathbf{x}}
\def\by{\mathbf{y}}
\def\bz{\mathbf{z}}

\def\bbP{\mathbb{P}}
\def\bbQ{\mathbb{Q}}
\def\bbR{\mathbb{R}}

\newcommand{\E}{\operatorname{\mathbb{E}}}
\renewcommand{\P}{\operatorname{\mathbb{P}}}

\def\Bin{\operatorname{Bin}}

\newcommand{\<}{\langle}
\renewcommand{\>}{\rangle}
\newcommand{\inner}[2]{\langle #1, #2 \rangle}

\newcommand{\xhat}{\widehat{x}}
\newcommand{\bxhat}{\widehat{\bx}}
\newcommand{\Shat}{\widehat{S}}

\definecolor{eac}{RGB}{200,50,50}
\definecolor{mad}{RGB}{50,200,50}
\definecolor{ejc}{RGB}{50,50,200}

\numberwithin{equation}{section}
\begin{document}

\title{On the Fundamental Limits of Adaptive Sensing}
\author{
Ery Arias-Castro%
\footnote{Department of Mathematics, University of California, San Diego
\{\href{mailto:eariasca@ucsd.edu}{eariasca@ucsd.edu}\}}, \
Emmanuel J.~Cand\`es%
\footnote{Departments of Mathematics and Statistics, Stanford University
\{\href{mailto:candes@stanford.edu}{candes@stanford.edu}\}} \
and Mark A.~Davenport
\footnote{School of Electrical and Computer Engineering, Georgia Institute of Technology
\{\href{mailto:mdav@gatech.edu}{mdav@gatech.edu}\}}
}

\date{November 2011 (Revised August 2012)}
\maketitle

\begin{abstract}
  Suppose we can sequentially acquire arbitrary linear measurements of
  an $n$-dimensional vector $\bx$ resulting in the linear model
  $\by = \bA \bx + \bz$, where $\bz$ represents measurement noise.  If the
  signal is known to be sparse, one would expect the following folk
  theorem to be true: choosing an {\em adaptive} strategy which
  cleverly selects the next row of $\bA$ based on what has been
  previously observed should do far better than a {\em nonadaptive}
  strategy which sets the rows of $\bA$ ahead of time, thus not trying
  to learn anything about the signal in between observations. This paper shows that the folk
  theorem is false.  We prove that the advantages offered by clever
  adaptive strategies and sophisticated estimation procedures---no
  matter how intractable---over classical compressed
  acquisition/recovery schemes are, in general, minimal.
\end{abstract}

\medskip


\section{Introduction}
\label{sec:intro}

This paper is concerned with the fundamental question of how well one
can estimate a sparse vector from noisy linear measurements in the
general situation where one has the flexibility to design those
measurements at will (in the language of statistics, one would say that
there is nearly complete freedom in designing the experiment).  This question is of importance in a variety of sparse signal estimation or sparse regression scenarios, but perhaps arises most naturally in the context of compressive sensing (CS)
\cite{MR2236170,MR2300700,donoho-CS}. In a nutshell, CS asserts that it is possible to reliably acquire sparse signals from just a few linear measurements selected a
priori. More specifically, suppose we wish to acquire a sparse signal $\bx
\in \bbR^n$.  A possible CS acquisition protocol would proceed as follows. $(i)$ Pick an $m \times n$ random projection matrix $\bA$ (the first $m$ rows of a random unitary matrix) in advance, and collect data of
the form
\begin{equation}
\label{measure}
\by = \bA \bx + \bz,
\end{equation}
where $\bz$ is a vector of errors modeling the fact that any real
world measurement is subject to at least a small amount of
noise. $(ii)$ Recover the signal by solving an $\ell_1$ minimization
problem such as the Dantzig selector~\cite{dantzig} or the
LASSO~\cite{Tibshirani96}.  As is now well known, theoretical results
guarantee that such convex programs yield accurate solutions.  In
particular, when $\bz = {\boldsymbol 0}$, the recovery is exact, and the error
degrades gracefully as the noise level increases.

A remarkable feature of the CS acquisition protocol is that the
sensing is completely nonadaptive; that is to say, no effort
whatsoever is made to understand the signal. One simply selects a
collection $\{\ba_i\}$ of sensing vectors a priori (the rows of the
matrix $\bA$), and measures correlations between the signal and these
vectors. One then uses numerical optimization---e.g., linear
programming \cite{dantzig}---to tease out the sparse signal $\bx$ from
the data vector $\by$. While this may make sense when there is no
noise, this protocol might draw some severe skepticism in a noisy
environment. To see why, note that in the scenario above, most of the power
is actually spent measuring the signal at locations where there is no
information content, i.e., where the signal vanishes. Specifically, let $\ba$ be a
row of the matrix $\bA$ which, in the scheme discussed above, has
uniform distribution on the unit sphere. The dot product is
\[
\<\ba, \bx\> = \sum_{j=1}^n a_j x_j,
\]
and since most of the coordinates $x_j$ are zero, one might think
that most of the power is wasted. Another way to express all of this
is that by design, the sensing vectors are approximately orthogonal to
the signal, yielding measurements with low signal power or a poor signal-to-noise ratio (SNR).

The idea behind adaptive sensing is that one
should localize the sensing vectors around locations where the signal
is nonzero in order to increase the SNR, or
equivalently, not waste sensing power. In other words, one should try
to ``learn'' as much as possible about the signal while acquiring it in
order to design more effective subsequent measurements. Roughly
speaking, one would $(i)$ detect those entries which are nonzero or
significant, $(ii)$ progressively localize the sensing vectors on those
entries, and $(iii)$ estimate the signal from such localized linear
functionals.  This is akin to the game of 20 questions in
which the search is narrowed by formulating the next question in a way
that depends upon the answers to the previous ones.  Note that in
some applications, such as in the acquisition of wideband radio
frequency signals, aggressive adaptive sensing mechanisms may not be practical
because they would require near instantaneous feedback.  However, there do exist
applications where adaptive sensing is practical and where the potential benefits of
adaptivity are too tantalizing to ignore.

The formidable possibilities offered by adaptive sensing give rise to
the following natural ``folk theorem.''
\begin{quote}
{\bf Folk Theorem}. {\em The estimation error one can get by
  using a clever adaptive sensing scheme is far better than what is
  achievable by a nonadaptive scheme. }
\end{quote}
In other words, learning about the signal along the way and
adapting the questions (the next sensing vectors) to what has been
learned to date is bound to help. In stark contrast, the main result
of this paper is this:
\begin{quote} {\bf Surprise}. {\em The folk theorem is wrong in
    general.  No matter how clever the adaptive sensing mechanism, no
    matter how intractable the estimation procedure, in general it is
    not possible to achieve a fundamentally better mean-squared error
    (MSE) of estimation than that offered by a na\"{i}ve random
  projection followed by $\ell_1$ minimization.}
\end{quote}

The rest of this article is mostly devoted to making this
claim precise. In doing so, we shall also show that adaptivity does
not help in obtaining a fundamentally better estimate of the
signal support, which is of independent interest.

\subsection{Main result}

To formalize matters, we assume that the error vector $\bz$ in
\eqref{measure} has i.i.d.\ $\cN(0,\sigma^2)$ entries. Then if $\bA$
is a random projection with unit-norm rows as discussed above,
\cite{dantzig} shows that the Dantzig selector estimate $ \bxhat^{\rm
  DS}$ (obtained by solving a simple linear program) achieves an MSE
obeying
\begin{equation}
\label{nonadapt}
\frac1n \E \|\bxhat^{\rm DS} - \bx \|_2^2 \leq C \, \frac{k}{m} \, \log(n)  \, \sigma^2,
\end{equation}
where $C$ is some numerical constant.  The bound holds {\em
  universally} over all $k$-sparse signals\footnote{A signal is said
  to be $k$-{\em sparse} if it has at most $k$ nonzero components.  We
  also occasionally use the notation $\|\bx\|_0$ to denote the number
  of nonzero components of $\bx$.}  provided that the number of
measurements $m$ is sufficiently large (on the order of at least $k
\log (n/k)$).  Moreover, one can show that this result is essentially
optimal in the sense that {\em any} possible nonadaptive choice of
$\bA$ (with unit-norm rows) and {\em any} possible estimation
procedure $\bxhat$ will satisfy
\begin{equation}
\label{nonadapt-LB}
\frac1n \E \|\bxhat - \bx\|_2^2 \geq C' \, \frac{k}{m} \, \log(n/k) \, \sigma^2,
\end{equation}
where $C'$ is a numerical constant~\cite{candes-davenport}.  The
fundamental question is thus: {\em how much lower can the MSE be} when
$(i)$ we are allowed to sense the signal adaptively and $(ii)$ we can
use any estimation algorithm we like to recover $\bx$.

The distinction between adaptive and nonadaptive sensing can be
expressed in the following manner.  Begin by rewriting the statistical model
\eqref{measure} as
\begin{equation}
\label{measure2}
y_i = \<\ba_i, \bx\> + z_i, \quad i = 1, \dots, m,
\end{equation}
in which a power constraint imposes that each $\ba_i$ is of norm at most 1,
i.e., $\|\ba_i\|_2 \leq 1$; then in a nonadaptive sensing scheme the
vectors $\ba_1, \dots, \ba_m$ are chosen in advance and do not depend
on $\bx$ or $\bz$ whereas in an adaptive setting, the measurement vectors may
be chosen depending on the history of the sensing process, i.e.,
$\ba_i$ is a (possibly random) function of $(\ba_1, y_1, \dots,
\ba_{i-1}, y_{i-1})$.

If we follow the principle that ``you cannot get something for
nothing,'' one might argue that giving up the freedom to adaptively select the sensing vectors would result in a
far worse MSE. Our main contribution is to show that this is not the
case.
\begin{thm}
\label{teo:main-minmax}
Suppose that $k < n/2$ and let $m$ be an arbitrary number of
measurements.  Assume that $\bx$ is sampled with i.i.d.\ coordinates
such that $x_j = 0$ with probability $1-k/n$ and $x_j = \mu$ with
probability $k/n$ (so that we have $k$ nonzero entries on the
average). Then for $\mu = \frac{4}{3} \sqrt{\frac{n}{m}}$, any sensing
strategy and any estimate $\bxhat$ obey \beq \label{main-minmax}
\frac1n \E \|\bxhat - \bx\|_2^2 \ge \frac{4}{27} \, \frac{k}{m} \,
\sigma^2 > \frac{1}{7} \, \frac{k}{m} \, \sigma^2.  \eeq
\end{thm}
For any $n$ and $k$, the number of nonzero entries in a random vector
drawn from the Bernoulli prior is between $k \pm 3\sqrt{k}$ with
probability at least 99\%. With some additional arguments, and when
$n$ and $k$ are sufficiently large, we can actually show that the last
inequality in \eqref{main-minmax} holds true in a minimax sense when
$\bx$ is known to have a support size in that range.  In order to
avoid unnecessary technicalities, we prove a simpler result.
\begin{thm}
  \label{thm:second-minmax}
  For any $n \ge 2$ and $k < n/2$, and any $m$, \beq \label{at-most-k}
  \inf_{\bxhat} \, \sup_{\|\bx\|_0 \le k} \ \frac1n \E \|\bxhat -
  \bx\|_2^2 \ge C_k \, \frac{k}{m} \, \sigma^2, \eeq in which $\inf_{k
    \ge 1} C_k \ge 1/33$. For $k \ge 10$ we can take $C_k \ge 1/15$, and when $k$ is sufficiently large we can take $C_k = 1/7$.
\end{thm}

In short, Theorems~\ref{teo:main-minmax} and~\ref{thm:second-minmax} say
that if one ignores a logarithmic factor, then {\em adaptive
  measurement schemes cannot (substantially) outperform nonadaptive
  strategies.} While seemingly counterintuitive, we find that
precisely the same sparse vectors which determine the minimax rate in
the nonadaptive setting are essentially so difficult to estimate that
by the time we have identified the support, we will have already
exhausted our measurement budget (i.e., we will have
  acquired all $m$ measurements).

Before moving on, we should clarify precisely what we mean by a {\em
    substantial} improvement. After all, the lower bound in
  Theorem~\ref{thm:second-minmax} does improve upon the nonadaptive
  bound in~\eqref{nonadapt-LB} by a factor of $\log(n/k)$.  Indeed, we will see in Section~\ref{sec:discussion} that at least in some very special cases (e.g., when $k=1$), this log factor can in fact be eliminated.   However,
  this is a relatively modest improvement compared to what one might
  hope to gain by exploiting adaptivity.  Specifically, consider a
  simple adaptive procedure that uses $m/2$ measurements to identify
  the support of $\bx$ and uses the remaining $m/2$ measurements to
  estimate the values of the nonzeros.  If such a scheme identifies
  the correct support, then it is easy to show that this procedure
  will yield an estimate satisfying
$$
\frac1n \E \|\bxhat - \bx\|_2^2 = \frac{2k}{n} \, \frac{k}{m} \, \sigma^2 .
$$
Thus, there seems to be room for reducing the error by a factor of $k/n$ beyond the $\log(n/k)$ factor.  Theorem~\ref{thm:second-minmax}, however, shows that this gain is not possible in general.

On the one hand, our main result states that one cannot universally
improve on bounds achievable via nonadaptive sensing
strategies. Indeed, we will see that there are natural classes of sparse signals for which,
even after applying the most clever sensing scheme and the most subtle
testing procedure, one would still not be sure about where the nonzeros
lie. This remains true even after having used up the entirety of our
measurement budget. On the other hand, our result does not say that adaptive
sensing {\em never} helps. In fact, there are many instances in which it
will.  For example, when some or most of the nonzero entries in $\bx$
are sufficiently large, they may be detected sufficiently early so
that one can ultimately get a far better MSE than what would be
obtained via a nonadaptive scheme, see Section \ref{sec:numerics} for
simple experiments in this direction and \secref{discussion} for further discussion.

\subsection{Connections with testing problems}

The arguments we develop to reach our conclusions are
quite intuitive, simple, and yet they seem different from the
classical Fano-type arguments for obtaining information-theoretic
lower bounds (see Section \ref{sec:fano} for a discussion of the
latter methods). Our approach involves proving a lower bound for the
Bayes risk under the prior from Theorem~\ref{teo:main-minmax}.
To obtain such a lower bound, we make a detour through
testing---multiple testing to be exact.
Our argument proceeds through two main steps:
\begin{itemize}
\item {\em Support recovery in Hamming distance.}  We consider the multiple
  testing problem of deciding which components of the signal are zero
  and which are not. We show that no matter which adaptive strategy
  and tests are used, the Hamming distance between the estimated and
  true supports is large. Put differently, the multiple testing
  problem is shown to be difficult. In passing, this establishes that
  adaptive schemes are not substantially better than nonadaptive
  schemes for support recovery.

\item {\em Estimation with mean-squared loss.}  Any estimator with a
  low MSE can be converted into an effective support
  estimator simply by selecting the largest coordinates or those above
  a certain threshold. Hence, a lower bound on the Hamming distance
  immediately gives a lower bound on the MSE.
\end{itemize}
The crux of our argument is thus to show that it is not possible to
choose sensing vectors adaptively in such a way that the support of
the signal may be estimated accurately.

\subsection{Differential entropies and Fano-type arguments}
\label{sec:fano}

Our approach is significantly different from classical methods for
getting lower bounds in decision and information theory.  Such methods
typically rely on Fano's inequality~\cite{MR2239987}, and
are all intimately related to methods in statistical decision theory
(see \cite{MR2724359,MR1462963}).  Before continuing, we would like to point out
that Fano-type arguments have been used successfully to obtain (often
sharp) lower bounds for some adaptive methods.  For example, the work
\cite{4494677} uses results from \cite{MR2724359} to establish a bound
on the minimax rate for binary classification (see the references
therein for additional literature on active learning).  Other examples
include the recent paper \cite{rigollet2010nonparametric}, which
derives lower bounds for bandit problems, and \cite{5394945} which
develops an information theoretic approach suitable for stochastic
optimization, a form of online learning, and gives bounds about the
convergence rate at which iterative convex optimization schemes
approach a solution.

Following the standard approaches in our setting leads to major obstacles that we would
like to briefly describe.  Our hope is that this will help the reader
to better appreciate our easy itinerary.  As usual, we start by
choosing a prior for $\bx$, which we take having zero mean.  Coming
from information theory, one would want to bound the mutual
information between $\bx$ (what we want to learn about) and $\by$ (the
information we have), for any measurement scheme $\ba_1, \dots,
\ba_m$.  Assuming a deterministic measurement scheme, by the chain
rule, we have
\begin{equation} \label{diff-entropy} I(\bx,\by) = h(\by) - h(\by \, |
  \, \bx) = \sum_{i = 1}^m h(y_i \, | y_{[i-1]}) - h(y_i \, |
  y_{[i-1]}, \bx),
\end{equation}
where $y_{[i]} := (y_1, \ldots, y_{i})$.  Since the history up to time
$i-1$ determines $\ba_i$, the conditional distribution of $y_i$ given
$y_{[i-1]}$ and $\bx$ is then normal with mean $\<\ba_i, \bx\>$ and
variance $\sigma^2$. Hence, $h(y_i \, |\, y_{[i-1]}, \bx) = \frac12
\log(2\pi e \sigma^2)$.
This is the easy term to handle---the challenging term
is $h(y_i \, |\, y_{[i-1]})$ and it is not clear how one should go
about finding a good upper bound. To see this, observe that
\[
\text{Var}(y_i \, | \, y_{[i-1]}) = \text{Var}(\<\ba_i, \bx\> \, | \,
y_{[i-1]}) + \sigma^2.
\]
A standard approach to bound $h(y_i \, |\, y_{[i-1]})$ is to write
\[
h(y_i \, | y_{[i-1]}) \le \frac12 \E \log \bigl(2\pi e
\text{Var}(\<\ba_i, \bx\> \, | \, y_{[i-1]}) + 2\pi
e \sigma^2\bigr),
\]
using the fact that the Gaussian distribution maximizes the entropy
among distributions with a given variance.  If we simplify the problem
by applying Jensen's inequality, we obtain
\beq \label{I} I(\bx,\by) \leq \sum_{i=1}^m \frac12
\log \bigl(\E \<\ba_i, \bx\>^2 / \sigma^2+ 1\bigr).
\eeq
The RHS needs
to be bounded uniformly over all choices of measurement schemes, which
is a daunting task given that $\ba_i$ is a function of
$y_{[i-1]}$ which is in turn a function of $\bx$.  We note however that the RHS {\em can} be bounded in the nonadaptive setting, which is the approach taken in~\cite{candes-davenport} to establish~\eqref{nonadapt-LB}. See also~\cite{raskutti2009minimax,5571873,verzelen2010minimax} for other asymptotic results in this direction.

We have presented the problem
in this form to help information theorists see the analogy with
the problem of understanding the role of feedback in a Gaussian
channel \cite{MR2239987}. Specifically, we can view the
inner products $\<\ba_i, \bx\>$ as inputs to a Gaussian channel where
we observe the output of the channel via feedback.  It is
well-known that feedback does not substantially increase the capacity
of a Gaussian channel, so one might expect this argument to be
relevant to our problem as well.  Crucially, however, in the case of a
Gaussian channel the user has full control over the channel
input---whereas in the absence of a priori knowledge of $\bx$, in our
problem we are much more restricted in our control over the ``channel
input'' $\< \ba_i, \bx \>$.

\subsection{Connections with other works}
\label{sec:connections}

A number of papers have studied the
advantages (or sometimes the lack thereof) offered by adaptive sensing in the setting where one has
{\em noiseless} data, see for example \cite{donoho-CS,NovakPower,indyk}
and references therein. Of course, it is well
known that one can uniquely determine a $k$-sparse vector from $2k$
linear nonadaptive noise-free measurements and, therefore, there is not
much to dwell on. The aforementioned works of course do not study such
a trivial problem. Rather, the point of view is that the signal is not
exactly sparse, only approximately sparse, and the question is thus
whether one can get a lower approximation error by employing an
adaptive scheme. Whereas we study a statistical problem, this is a
question in approximation theory. Consequently, the techniques and results of this line of research
have no bearing on our problem.

There is much research suggesting intelligent adaptive sensing
strategies in the presence of noise and we mention a few of these works.  In a setting closely
related to ours---that of detecting the locations of the nonzeros of a
sparse signal from noisy point samples (so that $m>n$)---\cite{haupt2009distilled} shows that by adaptively
allocating sensing resources one can significantly improve upon the
best nonadaptive schemes~\cite{dj04}.   Lower bounds for nonadaptive and adaptive methods in this context
were recently established in~\cite{malloy2011}, with the adaptive lower bounds established through the sequential probability ratio test (SPRT)~\cite{MR799155}.
Closer to
home,~\cite{haupt-adaptive,haupt-compressive} consider CS schemes (with
$m < n$) which perform sequential subset selection via the random
projections typical of CS, but which focus in on promising areas of
the signal. When the signal is $(i)$ very sparse $(ii)$ has sufficiently
large entries and $(iii)$ has constant dynamic range, the method in~\cite{haupt-compressive} is able to
remove a logarithmic factor from the MSE achieved by the Dantzig
selector
with (nonadaptive) i.i.d.\ Gaussian
measurements.
In a different direction,
\cite{4518814,4524050} suggest Bayesian approaches where the
measurement vectors are sequentially chosen so as to maximize the
conditional differential entropy of $y_i$ given $y_{[i-1]}$. Finally, another approach in~\cite{iwen} suggests a
bisection method based on repeated measurements for the detection of
1-sparse vectors, subsequently extended to $k$-sparse vectors via
hashing. None of these works, however, establish a lower bound on the MSE of the recovered signal.

\subsection{Content}

We prove all of our results in Section \ref{sec:main}, trying to give
as much insight as possible as to why adaptive methods are not much
more powerful than nonadaptive ones for detecting the support of a
sparse signal. We will also attempt to describe the regime in which
adaptivity might be helpful via simple numerical simulations in
Section \ref{sec:numerics}. These simulations show that adaptive algorithms are subject to a fundamental
phase transition phenomenon. Finally, we comment on
open problems and future research in Section
\ref{sec:discussion}.

\section{Limits of Adaptive Sensing Strategies}
\label{sec:main}

This section establishes nonasymptotic lower bounds for the estimation
of a sparse vector from adaptively selected noisy linear measurements.
To begin with, we remind ourselves that we collect possibly adaptive
measurements of the form \eqref{measure2} of an $n$-dimensional
signal $\bx$ where $\|\ba_i\|_2 \le 1$; from now on, we assume for simplicity and without loss
of generality that $\sigma = 1$.

In our analysis below, we denote the total-variation metric between any two probability
distributions $\bbP$ and $\bbQ$ by $\|\bbP - \bbQ\|_{\text{TV}}$, and their KL divergence
by $K(\bbP, \bbQ)$~\cite{MR2319879}. Our arguments
will make use of Pinsker's inequality, which relates these two
quantities via
\begin{equation}
  \label{pinsker} \|\bbP -
\bbQ\|_{\rm TV} \leq \sqrt{K(\bbQ, \bbP)/2}.
\end{equation}
We shall also use the convexity of the KL divergence, which states
that for $\lambda_i \ge 0$ and $\sum_i \lambda_i = 1$, we have
\begin{equation}
\label{eq:KLconvex}
K\Bigl(\sum_i \lambda_i \bbP_i, \sum_i \lambda_i \bbQ_i\Bigr) \leq \sum_i \lambda_i K(\bbP_i, \bbQ_i)
\end{equation}
in which $\{\bbP_i\}$ and $\{\bbQ_i\}$ are families of probability
distributions.

Before proceeding, we argue that when we are given a prior $\pi(\bx)$, we can restrict ourselves to deterministic measurement schemes in the sense that $\ba_1$ is a deterministic vector and, for $i \geq 2$, $\ba_i$ is a deterministic function of $y_{[i-1]} = (y_1, \dots, y_{i})$.  In the general case we have $\ba_i = F_i(y_{[i-1]}, U_i)$, where $F_i$ is a deterministic function and $U_i$ is random and independent of $y_{[i-1]}$ and $z_i$.  With $\mathbf{U} = (U_1, \ldots, U_m)$, it follows from the law of iterated expectation
\[
\E \|\bxhat - \bx\|^2 = \E \Bigl[ \E[\|\bxhat - \bx\|^2 | \mathbf{U}]
\Bigr]
\]
(the expectation in the left-hand side is taken over $\bx, \by$ and $\bU$) that there exists a fixed realization $\bu = (u_1, \ldots, u_m)$ obeying
\[
\E[\|\bxhat - \bx\|^2 | \mathbf{U} = \bu] \le \E \|\bxhat - \bx\|^2.
\]
Hence, we can construct an estimator based on a deterministic
measurement scheme which is as good as any based on a randomized
measurement scheme. Note that in a deterministic scheme, letting
$\bbP_\bx$ be the distribution of $y_{[i-1]}$ when the target vector
is $\bx$ and using the fact that $y_i$ is conditionally independent of
$y_{[i-1]}$ given $\ba_i$, we see that the likelihood factorizes as
\begin{equation}
 \label{Px}
\P_\bx(y_{[m]}) = \prod_{i=1}^m \P_\bx(y_i | \ba_i),
\end{equation}
which will be of use in our analysis below.

\subsection{The Bernoulli prior}
\label{sec:bernoulli}

We begin by studying the model in Theorem~\ref{teo:main-minmax} which
makes our argument most transparent. The proof of
Theorem~\ref{thm:second-minmax} essentially reduces to that of
Theorem~\ref{teo:main-minmax}.

In this model, we suppose that $\bx \in \bbR^n$ is sampled from a product prior: for each $j \in \{1, \ldots, n\}$,
\beq \label{bernoulli}
x_j = \begin{cases} 0 & \text{w.p. } 1-k/n,\\
\mu & \text{w.p. } k/n,
\end{cases}
\eeq
and the $x_j$'s are independent. In this model, $\bx$ has on
average $k$ nonzero entries, all with known positive amplitudes equal
to $\mu$. This model is easier to study than the related model in which
one selects $k$ coordinates uniformly at random and sets those to
$\mu$. The reason is that in this Bernoulli model, the independence
between the coordinates of $\bx$ brings welcomed simplifications, as
we shall see.

Our goal here is to establish a lower bound on the MSE
when $\bx$ is drawn from this prior.  We do this in two steps.  First,
we look at recovering the support of $\bx$, which is done via a
reduction to multiple testing.  Second, we show that a lower bound on
the error for support recovery implies a lower bound on the MSE,
leading to Theorem~\ref{teo:main-minmax}.

\subsubsection{Support recovery in Hamming distance}

We would like to understand how well we can estimate the support $S =
\{j : x_j \neq 0\}$ of $\bx$ from the data \eqref{measure2}, and shall
measure performance by means of the expected Hamming distance. Here,
the error of a procedure $\Shat$ for estimating the support $S$ is
defined as
\[
\E |\Shat \Delta S|  = \sum_{j=1}^n \P( \Shat_j \neq S_j )
\]
where $\Delta$ denotes the symmetric difference, $S_j = 1$ if $j \in S$ and equals zero otherwise, and similarly for $\Shat_j$. As we can see, this reduces our problem to a sequence of $n$ independent hypothesis tests.  We will obtain a lower bound on the number of errors among these tests by exploiting the following lemma.
\begin{lem} \label{lem:Bayes}
Consider the testing problem of deciding between $H_0 : \bx \sim \P_0$ and $H_1 : \bx \sim \P_1$, where $H_0$ and $H_1$ occur with prior probabilities $\pi_0$ and $\pi_1$ respectively.  Under the 0-1 loss, The Bayes risk $B$ obeys
$$
B \ge \min(\pi_0, \pi_1) \left(1 - \| \P_1 - \P_0 \|_{\mathrm{TV}} \right).
$$
\end{lem}
\begin{proof}
Assume without loss of generality that $\pi_1 \le \pi_0$. The test with minimum risk is the Bayes test rejecting $H_0$ if and only if
$$
\Lambda = \frac{\pi_1 \, \bbP_{1}(\bx)}{\pi_0 \, \P_{0}(\bx)} > 1;
$$
that is, if the adjusted likelihood ratio exceeds one; see~\cite[Pbm. 3.10]{TSH}. A simple calculation shows that the Bayes risk obeys
$$
B = \pi_0 \E_0 \left( \min(1,\Lambda) \right),
$$
where $\E_0$ denotes expectation under $\P_{0}$.  Using the fact that $\E_0 \Lambda = \pi_1/\pi_0$ together with
$$
\min(1,\Lambda) = \frac{1+\Lambda}{2} + \frac{|\Lambda-1|}{2},
$$
we obtain
\begin{equation} \label{eq:Bayes1}
B = \frac{1}{2} - \frac{\pi_0}{2} \E_0 | \Lambda-1 |.
\end{equation}
Finally,
\begin{align*}
  \pi_0 \E_0 |\Lambda-1| = \int | \pi_1 \text{d}\P_1 - \pi_0 \text{d}\P_0 | & \le \pi_1 \int |\text{d}\P_1 - \text{d}\P_0| + \pi_0 - \pi_1 \\
  & = 2\pi_1 \| \P_1 - \P_0 \|_{\mathrm{TV}} + \pi_0 - \pi_1,
\end{align*}
which when combined with~\eqref{eq:Bayes1} establishes the lemma.
\end{proof}

\begin{thm} \label{thm:bernoulli-support}
Suppose that $\bx$ is sampled according to the Bernoulli prior with $k \le n/2$, then any estimate $\Shat$ obeys
\begin{equation}
\label{eq:supportH}
\E |\Shat \Delta S| \geq k \Bigl(1 - \frac{\mu}{2} \sqrt{\frac{m}{n}}\Bigr).
\end{equation}
\end{thm}
Hence, if the amplitude of the signal is below $\sqrt{n/m}$, we expect a large number of errors; indeed, if $\mu = \sqrt{n/m}$, then
$\E |\Shat \Delta S| \ge k/2$.

\begin{proof}\footnote{The main ideas of our proof are
      similar to those in that of Assouad's Lemma, see
      \cite{MR777600,MR2724359} for instance.  Note, however, that our
      approach yields a sharper constant.}  Let $\pi_1 =
    k/n$ and $\pi_0 = 1 - \pi_1$. For any $j$, set $\bbP_{0,j} =
    \P(\cdot | x_j = 0)$ and $\bbP_{1,j} = \P(\cdot | x_j \neq
    0)$. Let $B_j$ denote the Bayes risk of the decision problem
    $H_{0,j} : x_j = 0$ versus $H_{1,j} : x_j = 1$.  From
    Lemma~\ref{lem:Bayes} we have that
$$
\E |\Shat \Delta S| = \sum_{j=1}^n \P(\Shat_j \neq S_j) \ge  \sum_{j=1}^n  B_j \ge \pi_1 \sum_{j=1}^n \Bigl(1-\|\bbP_{1,j} - \bbP_{0,j}\|_{\text{TV}}\Bigr).
$$
Applying the Cauchy-Schwartz inequality, we obtain
\begin{equation} \label{eq:intermediate_a}
\E |\Shat \Delta S| \ge k \Bigl(1- \frac{1}{\sqrt{n}} \sqrt{\sum_{j=1}^n \|\bbP_{1,j} - \bbP_{0,j}\|_{\text{TV}}^2}\Bigr).
\end{equation}
The theorem is a consequence of~\eqref{eq:intermediate_a} combined with
\begin{equation} \label{eq:intermediate}
  \sum_{j=1}^n \|\bbP_{1,j} -   \bbP_{0,j}\|_{\text{TV}}^2 \le \frac{\mu^2}{4} \, m.
\end{equation}

To establish~\eqref{eq:intermediate}, we apply Pinsker's inequality twice to obtain
\begin{equation}
\label{XLi}
\|\bbP_{1,j} - \bbP_{0,j}\|^2_{\text{TV}} \le \frac{\pi_0}{2}
K(\bbP_{0,j},\bbP_{1,j}) + \frac{\pi_1}{2} K(\bbP_{1,j},\bbP_{0,j})
\end{equation}
so that it remains to find an upper bound on the KL divergence between
$\bbP_{0,j}$ and $\bbP_{1,j}$. Write $\bbP_0 = \bbP_{0,j}$ for short
and likewise for $\bbP_{1,j}$. Then
\[
\P_0(y_{[m]}) = \sum_{\bx'} \P(\bx') \P(y_{[m]} | x_j = 0, \bx') :=
\sum_{\bx'} \P(\bx') \bbP_{0,\bx'},
\]
where $\bx' = (x_1, \ldots, x_{j-1}, x_{j+1}, \ldots, x_n)$ and
$\bbP_{0,\bx'}$ is the conditional probability
  distribution of $y_{[m]}$ given $\bx'$ and $x_j = 0$;
  $\P_1(y_{[m]})$ is defined similarly.  The convexity of the KL
divergence \eqref{eq:KLconvex} gives \beq \label{KL} K(\bbP_0,\bbP_1)
\le \sum_{\bx'} \P(\bx') K(\bbP_{0,\bx'}, \bbP_{1,\bx'}).  \eeq We now
calculate this divergence. In order to do this, observe that we have
$y_i = \<\ba_i,\bx\> + z_i = c_i + z_i$ under $\bbP_{0,\bx'}$ while
$y_i = a_{i,j} \mu + c_i + z_i$ under $\bbP_{1,\bx'}$. This yields
\begin{align*}
K(\bbP_{0,\bx'}, \bbP_{1,\bx'})
& = \E_{0,\bx'} \log \frac{\bbP_{0,\bx'}}{\bbP_{1,\bx'}}\\
&= \sum_{i=1}^m \E_{0,\bx'} \left(\frac12 (y_i  - \mu a_{i,j}  - c_i )^2 - \frac12 (y_i  - c_i )^2\right) \\
&= \sum_{i=1}^m \E_{0,\bx'} \left(- z_i \mu a_{i,j} + (\mu a_{i,j})^2/2\right)\\
&= \frac{\mu^2}2 \sum_{i=1}^m \E_{0,\bx'} (a_{i,j}^2).
\end{align*}
The first equality holds by definition, the second follows from
\eqref{Px}, the third from $y_i = c_i + z_i $ under $\bbP_{0,\bx'}$
and the last holds since $z_i$ is independent of $a_{i,j}$ and has
zero mean. Using \eqref{KL}, we obtain
\[
K(\bbP_{0},\bbP_{1}) \le \frac{\mu^2}2 \sum_{i=1}^m \E[ a_{i,j}^2 |
x_j = 0].
\]
Similarly,
\[
K(\bbP_{1},\bbP_{0}) \le \frac{\mu^2}2 \sum_{i=1}^m \E[ a_{i,j}^2 |
x_j = \mu]
\]
and, therefore, \eqref{XLi} shows that
\[
\|\bbP_{1,j} - \bbP_{0,j}\|^2_{\text{TV}} \le \frac{\mu^2}{4}
\Bigl(\sum_{i=1}^m \pi_0 \E[ a_{i,j}^2 | x_j = 0] + \pi_1 \E[
a_{i,j}^2 | x_j = \mu] \Bigr) = \frac{\mu^2}{4} \sum_{i=1}^m \E[
a_{i,j}^2].
\]
For any particular pair $(i,j)$ with $i>1$, we can say
  very little about $\E[ a_{i,j}^2]$ since it can depend on all the
  previous measurements in a potentially very complicated manner.
  However, by summing this inequality over $j$ we can
  obtain~\eqref{eq:intermediate} by using the only constraint we have
  imposed on the $\ba_i$, namely, $\|\ba_i\|_2 = 1$, so that
  $\sum_{ij} a^2_{ij} = m$.  This establishes the theorem.
\end{proof}

\subsubsection{Estimation in mean-squared error}
\label{sec:estimation}

It is now straightforward to obtain a lower bound on the MSE from
Theorem \ref{thm:bernoulli-support}.
\begin{proof}[Proof of Theorem~\ref{teo:main-minmax}]
  Let $S$ be the support of $\bx$ and set $\Shat := \{j: |\xhat_j|
  \geq \mu/2\}$.  We have
\[
\|\bxhat - \bx\|_2^2 = \sum_{j \in S} (\xhat_j - x_j)^2 + \sum_{j
  \notin S} \xhat_j^2 \geq \frac{\mu^2}{4} |S \setminus \Shat| +
\frac{\mu^2}{4} |\Shat \setminus S| = \frac{\mu^2}{4} |\Shat \Delta S|
\]
and, therefore,
\[
\E \|\bxhat - \bx\|_2^2 \geq \frac{\mu^2}{4} \E
|\Shat \Delta S| \geq \frac{\mu^2}{4} k \Bigl(1 - \frac{\mu}{2}
\sqrt{\frac{m}{n}}\Bigr),
\]
where the last inequality is from \thmref{bernoulli-support}.  We then
plug in $\mu = \frac{4}{3}\, \sqrt{\frac{n}{m}}$ and simplify to conclude.
\end{proof}

\subsection{The conditional Bernoulli prior and minimax bound}
\label{sec:minmax}

To establish Theorem~\ref{thm:second-minmax}, we choose as
distribution on $\bx$ the prior $\nu_{n,k}$ defined as follows: we
start with the Bernoulli prior $\pi_{n, \alpha k}$ \eqref{bernoulli}
with mean $\alpha k$ (instead of $k$) for some fixed $\alpha \in
(0,1)$, and then condition that distribution to realizations with at
most $k$ nonzero entries.

\begin{prp}
\label{prp:support2}
Suppose that $\bx$ is sampled according to $\nu_{n,k}$ with $k\le n/2$, then any estimate $\Shat$ obeys
\begin{equation}
\label{eq:supportH2}
\E |\Shat \Delta S| \geq \alpha k \Bigl(1-\gamma_{n,k}(\alpha)  -  \frac{\mu}2 \sqrt{\frac{m}{n}}\Bigr),
\end{equation}
where
\[
\gamma_{n,k}(\alpha) := \frac1\alpha \sum_{j = k+1}^n (2 + (j-1)/k)
\P(\Bin(n, \alpha k/n) = j).
\]
\end{prp}
\begin{proof}
  We begin by arguing that we can restrict attention to estimates
  $\Shat$ with cardinality at most $2k-1$. To see why, consider an
  arbitrary estimate $\Shat$ and set
\[
\Shat_k = \begin{cases}
\Shat, & |\Shat| \le 2 k -1, \\
\emptyset, & |\Shat| \ge 2 k.
\end{cases}
\]
Now if $|\Shat| \ge 2 k$, then for any $S$ with $|S| \le k$, we have
\[
|\Shat \Delta S| \ge |\Shat \setminus S| \ge |\Shat| - |S| \ge k \ge
|\emptyset \Delta S|.
\]
Since $|S| \le k$ under $\nu_{n,k}$, it follows that $\E |\Shat \Delta
S| \ge \E |\Shat_k \Delta S|$, which proves the claim. From now on, we
assume that $|\Shat| < 2k$.

Set $\pi_{n, \alpha k}(k) = \P_{\bx \sim \pi_{n,\alpha k}}(|S| \le
k)$ and observe the identity
\[
\E_{\bx \sim \nu_{n,k}} |\Shat \Delta S| = \E_{\bx \sim \pi_{n,\alpha k}}
\left[ |\Shat \Delta S| \, \vert \, |S| \le k \right] =
\frac{1}{\pi_{n, \alpha k}(k)} \, \E_{\bx \sim \pi_{n,\alpha k}}
\left[ |\Shat \Delta S| \, {\bf 1}_{\{|S| \le k\}} \right].
\]
To conclude, \thmref{bernoulli-support} together with $|\Shat \Delta
S| \le |\Shat| + |S| \le 2k-1 + |S|$ give
\begin{align*}
  \pi_{n, \alpha k}(k) \, \E_{\bx \sim \nu_{n,k}} |\Shat \Delta S| &
  = \E_{\bx \sim \pi_{n,\alpha k}} |\Shat \Delta S| - \E_{\bx \sim \pi_{n,\alpha k}} |\Shat \Delta S| \ {\bf 1}_{\{|S| \ge k+1\}}  \\
  & \ge \alpha k \left(1 - \frac\mu2 \sqrt{\frac{m}n}\right) - \sum_{j
    = k+1}^n (2k -1 + j) \P(\Bin(n, \alpha k/n) = j).
\end{align*}
\end{proof}

We do as in \secref{estimation} to conclude the proof of
Theorem~\ref{thm:second-minmax}.  Let $\gamma$ be a short for
$\gamma_{n,k}(\alpha)$.  We find that the optimal choice is $\mu =
\frac43 (1 - \gamma) \sqrt{n/m}$, yielding the lower bound
\beq \label{MSE1} \E \|\bxhat - \bx\|_2^2 \ge \alpha (1 - \gamma)
\frac{4 k}{27}\sqrt{n/m}.  \eeq
To obtain a bound on $\gamma$, note that we can write
\[
\alpha k \, \gamma_{n,k}(\alpha) = 3k \P(\Bin(n, \alpha k/n) \ge k+1) + \sum_{j \ge k+2} \P(\Bin(n, \alpha k/n) \ge j).
\]
Bennett's inequality applied to the binomial distribution, gives
\beq \label{bennett}
\P(\Bin(m, p) \ge j) \le \exp\big[- j \log(j/(mp)) +j - mp \big].
\eeq
Therefore, if $j \ge k$, we have
\[
\P(\Bin(n, \alpha k/n) \ge j) \le \exp\big[- j \log(j/(\alpha k)) +j - \alpha k \big] \le \exp[- \beta j], \quad \beta := \alpha - 1 - \log \alpha,
\]
where the last inequality follows from the fact that the exponent is increasing in $k$ over the range $(0, j/\alpha)$.  Note that $\beta > 0$ for any $\alpha < 1$.  Applying this inequality, we get
\beq \label{gamma}
\alpha k \, \gamma_{n,k}(\alpha) \le 3 k e^{- (k+1) \beta} + \frac{e^{- (k+2) \beta}}{1 - e^{-\beta}} \le (3 k+1) e^{- (k+1) \beta},
\eeq
when $\beta \ge \log 2$.  This bound yields $\gamma < 1$
for all $k \ge 1$ when $\alpha \le 0.03$, in which case \eqref{eq:supportH2} and \eqref{MSE1} become meaningful.
This bound is quite conservative, however.  Using the definition of $\gamma$, we can numerically show that that by choosing $\alpha$ appropriately we can obtain $\alpha (1
- \gamma) \ge 2 e^{-1/2}-1 \ge 0.21$ for any $n \ge 2$ and all $k \le n/2$.  Thus we can always write $\alpha (1 - \gamma) \frac4{27} \ge \frac1{33}$.  While setting $C_k = \frac1{33}$ ensures that Theorem~\ref{thm:second-minmax} holds for any possible choice of $k$, it is somewhat pessimistic in the sense that it is entirely dictated by the special case of $k=1$ (which could be handled more efficiently by alternative means~\cite{adaptiveISIT}).  For larger values of $k$, it is possible to obtain an improved constant.  For example, when $k \ge 10$ numerical calculations show that we can take $C_k = \frac1{15}$.  Moreover, in view of the first inequality in~\eqref{gamma}, and the fact that $\beta > 0$ for all $\alpha < 1$, we have $\gamma_{n,k}(\alpha) \to 0$ as $k \to \infty$ and $\alpha$ is held fixed.  Thus, for $\alpha$ sufficiently close to $1$ we will have that $\alpha(1-\gamma) \frac4{27} \ge \frac17$ for $k$ sufficiently large.  We have also verified this numerically.  Hence, the numerical constant $\frac17$ of Theorem~\ref{teo:main-minmax} is also valid in
Theorem~\ref{thm:second-minmax} provided $k$ is sufficiently large.

\section{Numerical Experiments}
\label{sec:numerics}

In order to briefly illustrate the implications of the lower bounds
in~\secref{main} and the potential limitations and benefits of
adaptivity in general, we include a few simple numerical
experiments. To simplify our discussion, we limit ourselves to
existing adaptive procedures that aim at consistent support recovery:
the adaptive procedure from~\cite{4518814} and the recursive bisection
algorithm of~\cite{iwen}.

We emphasize that in the case of a generic $k$-sparse signal, there
are many possibilities for adaptively estimating the support of the
signal.  For example, the approach in~\cite{haupt-compressive}
iteratively rules out indices and could, in principle, proceed until
only $k$ candidate indices remain.  In contrast, the approaches
in~\cite{4518814} and~\cite{iwen} are built upon algorithms for
estimating the support of $1$-sparse signals.  An algorithm for a
$1$-sparse signal could then be run $k$ times to estimate a $k$-sparse
signal as in~\cite{4518814}, or used in conjunction with a hashing
scheme as in~\cite{iwen}.  Since our goal is not to provide a thorough
evaluation of the merits of all the different possibilities, but
merely to illustrate the general limits of adaptivity, we simplify our
discussion and focus exclusively on the simple case of one-sparse
signals, i.e., where $k=1$.

Specifically, in our experiments we will consider the uniform prior on
the set of vectors with a single nonzero entry equal to $\mu > 0$ as
in \secref{main}.  Since we are focusing only on the case of $k=1$,
the algorithms in~\cite{4518814} and~\cite{iwen} are extremely simple
and are shown in \algref{1} and \algref{2} respectively. Note that in
\algref{1} the step of updating the posterior distribution $\bp$
consists of an iterative update rule given in~\cite{4518814} and does
not require any a priori knowledge of the signal $\bx$ or $\mu$.  In
\algref{2}, we simplify the recursive bisection algorithm
of~\cite{iwen} using the knowledge that $\mu > 0$, which allows us to
eliminate the second stage of the algorithm aimed at detecting
negative coefficients.  Note that this algorithm proceeds through
$s_{\mathrm{max}} = \log_2 n$ stages and we must allocate a certain
number of measurements to each stage.  In our experiments we set $m_s
= \lceil \beta 2^{-s} \rceil$, where $\beta$ is selected to ensure
that $\sum_{s=1}^{\log_2 n} m_s \le m$.

\begin{algorithm}[t]
\caption{Adaptive algorithm from~\cite{4518814}} \label{alg:1}
\begin{algorithmic}
\STATE \textbf{input:} $m \times n$ random matrix $\bB$ with i.i.d.\ Rademacher ($\pm 1$ with equal probability) entries.
\STATE \textbf{initialize:} $\bp = \frac{1}{n}(1, \ldots, 1)^T$.
\FOR{$i=1$ to $i = m$}
\STATE Compute $\ba_i  = (b_{i,1} \sqrt{p_1}, \ldots, b_{i,n} \sqrt{p_n})^T$.
\STATE Observe $y_i = \inner{\ba_i}{\bx} + z_i$.
\STATE Update posterior distribution $\bp$ of $\bx$ given $(\ba_1,y_1), \ldots, (\ba_i,y_i)$ using the rule in~\cite{4518814}.
\ENDFOR
\STATE \textbf{output:} Estimate for $\mathrm{support}(\bx)$ is the index where $\bp$ attains its maximum value.
\end{algorithmic}
\end{algorithm}

\begin{algorithm}[t]
\caption{Recursive bisection algorithm of~\cite{iwen}} \label{alg:2}
\begin{algorithmic}
\STATE \textbf{input:} $m_1$, \ldots, $m_{s_{\mathrm{max}}}$.
\STATE \textbf{initialize:} $J_1^{(1)} = \{1, \ldots, \frac{n}{2}\}$, $J_2^{(1)} = \{ \frac{n}{2}+1, \ldots, n\}$.
\FOR{$s=1$ to $s = s_{\mathrm{max}}$}
\STATE Construct the $m_s \times n$ matrix $\bA^{(s)}$ with rows $|J_1^{(s)}|^{-\frac12}{\bf 1}_{J_1^{(s)}} - |J_2^{(s)}|^{-\frac12}{\bf 1}_{J_2^{(s)}}$.
\STATE Observe $\by^{(s)} = \bA^{(s)} \bx + \bz^{(s)}$.
\STATE Compute $w^{(s)} = \sum_{i=1}^{m_s} y_i^{(s)}$.
\STATE Subdivide: Update $J_1^{(s+1)}$ and $J_2^{(s+1)}$ by partitioning $J_1^{(s)}$ if $w^{(s)} \ge 0$ or $J_2^{(s)}$ if $w^{(s)} < 0$.
\ENDFOR
\STATE \textbf{output:} Estimate for $\mathrm{support}(\bx)$ is $J_1^{(s_{\mathrm{max}})}$ if $w^{(s_{\mathrm{max}})} \ge 0$, $J_2^{(s_{\mathrm{max}})}$ if $w^{(s_{\mathrm{max}})} < 0$.
\end{algorithmic}
\end{algorithm}

\subsection{Evolution of the posterior}

\begin{figure}[t]
   \centering
   \begin{tabular}{ccc}
   \hspace{-3mm} \includegraphics[width=.33\linewidth]{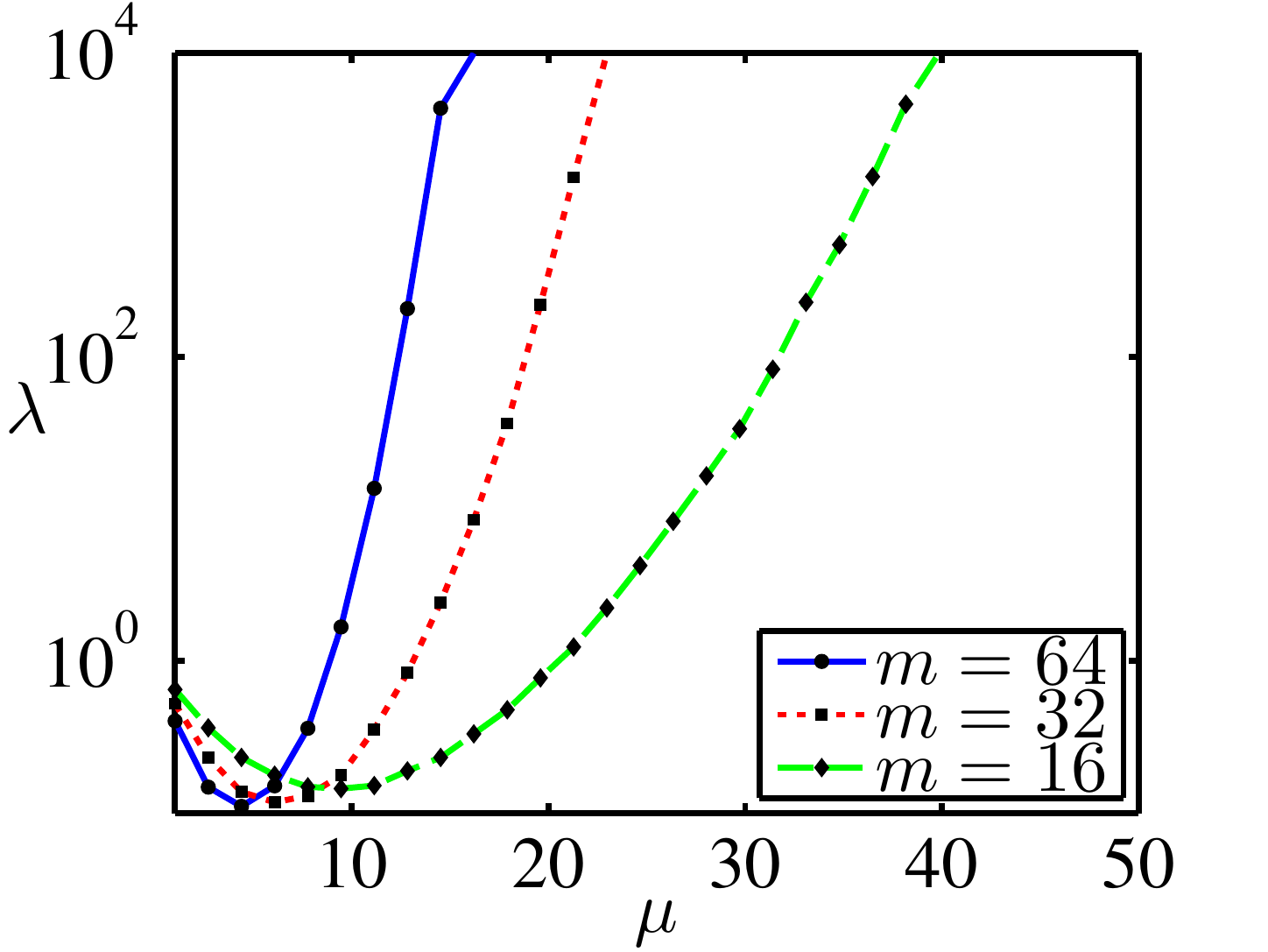} & \hspace{-3mm} \includegraphics[width=.33\linewidth]{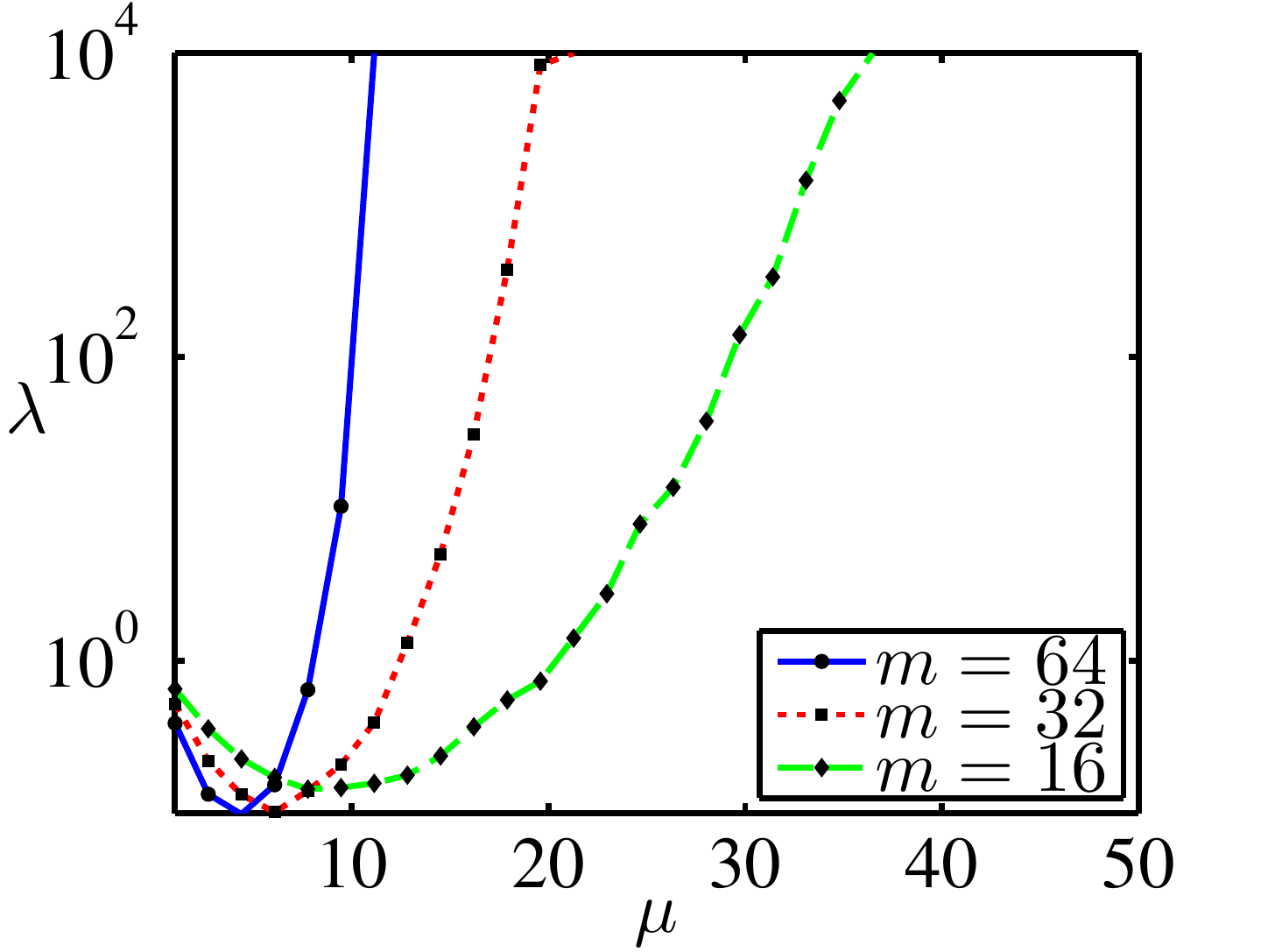} & \hspace{-3mm} \includegraphics[width=.33\linewidth]{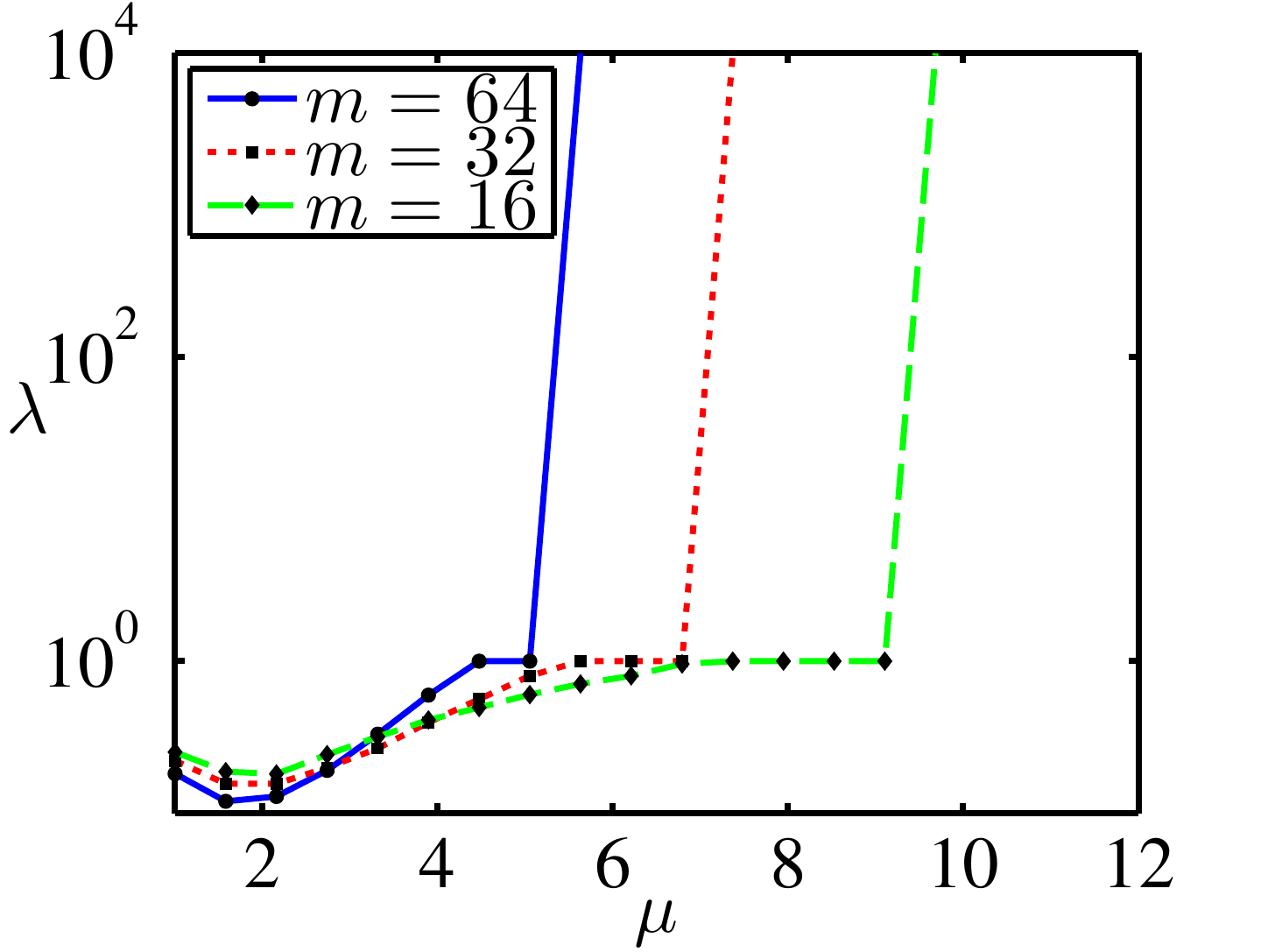} \\
   \hspace{-2mm} {\small \sl (a)} &  \hspace{-2mm} {\small \sl (b)} & \hspace{-2mm} {\small \sl  (c)}
   \end{tabular}
   \caption{\small \sl Behavior of the posterior distribution as a function of $\mu$ for several values of $m$. (a) shows the results for nonadaptive measurements.  (b) shows the results for~\algref{1}.  (c) shows the results for~\algref{2}.  We see that~\algref{2} is able to detect somewhat weaker signals than~\algref{1}.  However, for both cases we observe that once $\mu$ exceeds a certain threshold proportional to $\sqrt{n/m}$, the ratio $\lambda$ of $p_{j^*}$ to the second largest posterior probability grows exponentially fast, but that this does not differ substantially from the behavior observed in (a) when using nonadaptive measurements.
     \label{fig:figure1}}
\end{figure}

We begin by showing the results of a simple simulation that
illustrates the behavior of the posterior distribution of $\bx$ as a function of $\mu$ for
both adaptive schemes.  Specifically, we assume that $m$ is fixed and collect $m$ measurements using each approach.  Given the measurements $\by$, we then compute the posterior distribution $\bp$ using the true prior used to generate the signal, which can be computed using the fact that
\begin{equation} \label{eq:postCompute}
p_j \propto \exp \left( - \frac{1}{2\sigma^2} \| \by - \mu \bA \be_j \|_2^2 \right),
\end{equation}
where $\sigma^2$ is the noise variance and $\be_j$ denotes the $j$th element of the standard basis.
What we expect is that once $\mu$ exceeds a certain threshold (which depends
on $m$), the posterior will become highly concentrated on the true
support of $\bx$.  To quantify this, we consider the case where $j^*$
denotes the true location of the nonzero element of $\bx$ and define
\[
\lambda  =  \frac{p_{j^*}}{ \max_{j \neq j^*} p_j }.
\]
Note that when $\lambda \le 1$, we cannot reliably detect the nonzero,
but when $\lambda \gg 1$ we can.

In~\figref{figure1} we show the results for a few representative values of $m$ (a) when using nonadaptive measurements, i.e., a (normalized) i.i.d.\ Rademacher random matrix $\bA$, compared to the results of (b) \algref{1}, and (c) \algref{2}.  For each value of $m$ and for each
value of $\mu$, we acquire $m$ measurements using each approach and compute
the posterior $\bp$ according to~\eqref{eq:postCompute}.  We then compute the value of
$\lambda$.  We repeat this for 10,000 iterations and plot the median
value of $\lambda$ for each value of $\mu$ for all three approaches.
In our experiments we set $n = 512$ and $\sigma^2 = 1$.  We truncate
the vertical axis at $10^4$ to ensure that all curves are comparable.
We observe that in each case, once $\mu$ exceeds a certain threshold
proportional to $\sqrt{n/m}$, the ratio $\lambda$ of
$p_{j^*}$ to the second largest posterior probability grows
exponentially fast.  As expected, this occurs for both the nonadaptive and adaptive strategies, with no substantial difference in terms of how large $\mu$ must be before support recovery is assured (although~\algref{2} seems to improve upon the nonadaptive strategy by a small constant).

\subsection{MSE performance}

\begin{figure}[t]
   \centering
   \includegraphics[width=.5\linewidth]{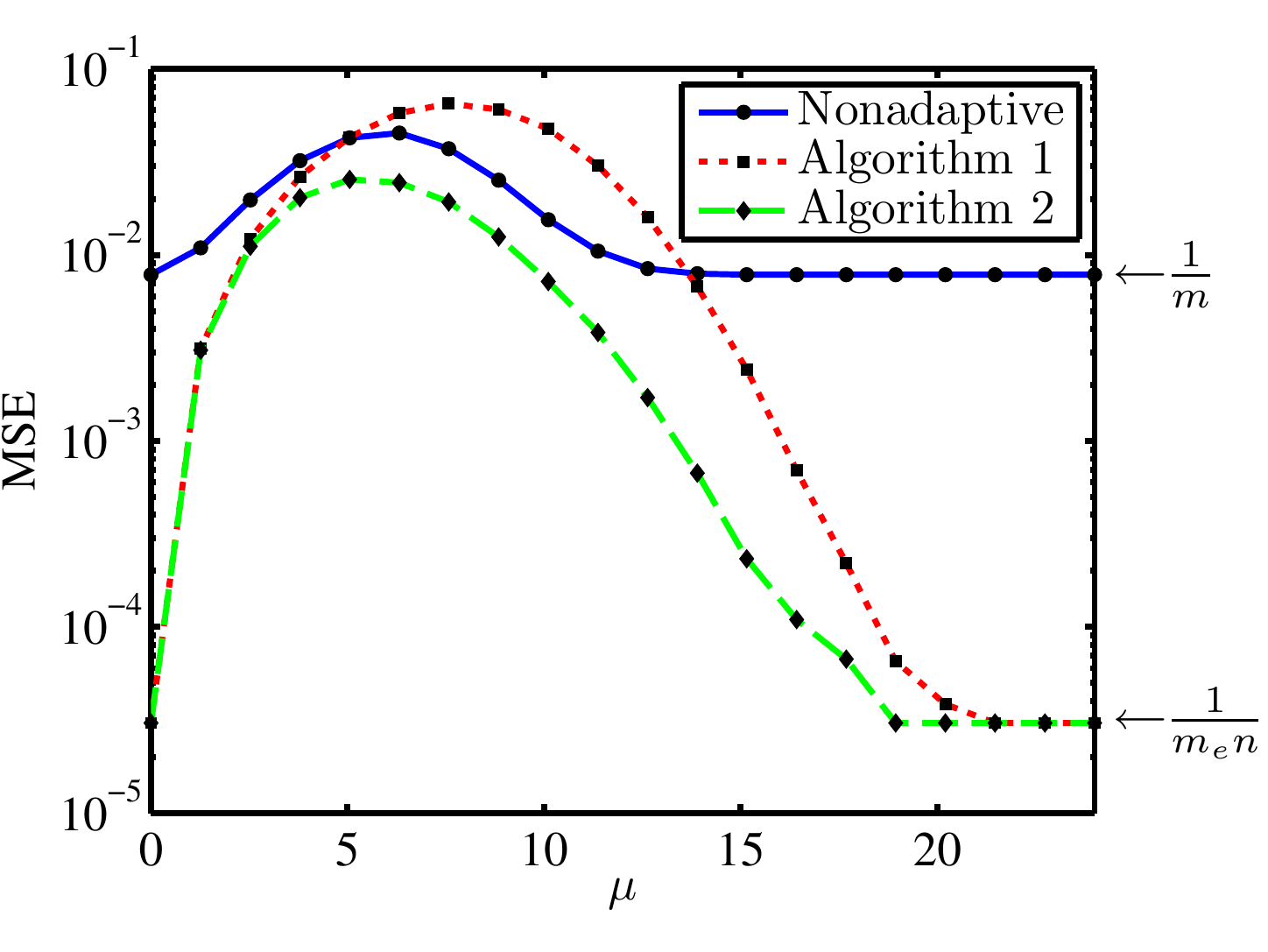}
   \caption{\small \sl The performance of \algref{1} and \algref{2} in
     the context of a two-stage procedure that first uses $m_d =
     \frac{m}{2}$ adaptive measurements to detect the location of the
     nonzero and then uses $m_e = \frac{m}{2}$ measurements to
     directly estimate the value of the identified coefficient.  We
     show the resulting MSE as a function of the amplitude $\mu$ of
     the nonzero entry, and compare this to a nonadaptive procedure
     which uses a (normalized) i.i.d.\ Rademacher matrix followed by
     OMP.  In the worst case, the MSE of the adaptive algorithms is
     comparable to the MSE obtained by the nonadaptive algorithm and
     exceeds the lower bound in Theorem~\ref{thm:second-minmax} by
     only a small constant factor.  When $\mu$ begins to exceed this
     critical threshold, the MSE of the adaptive algorithms rapidly
     decays below that of the nonadaptive algorithm and approaches
     $\frac{1}{m_e n}$, which is the MSE one would obtain given $m_e$
     measurements and a priori knowledge of the support.
     \label{fig:figure2}}
\end{figure}

We have just observed that for a given number of measurements $m$,
there is a critical value of $\mu$ below which we cannot reliably
detect the support.  In this section we examine the impact of this
phenomenon on the resulting MSE of a two-stage procedure that first
uses $m_d = pm$ adaptive measurements to detect the location of the
nonzero with either~\algref{1} or~\algref{2} and then reserves $m_e =
(1-p)m$ measurements to directly estimate the value of the identified
coefficient.  It is not hard to show that if we correctly identify the
location of the nonzero, then this will result in an MSE of $(m_e
n)^{-1} = ((1-p)mn)^{-1}$.  As a point of comparison, if an oracle
provided us with the location of the nonzero a priori, we could devote
all $m$ measurements to estimating its value, with the best possible
MSE being $\frac{1}{m n}$.  Thus, if we can correctly detect the
nonzero, this procedure will perform within a constant factor of the
oracle.

We illustrate the performance of~\algref{1} and~\algref{2} in terms of
the resulting MSE as a function of the amplitude $\mu$ of the nonzero
in~\figref{figure2}.  In this experiment we set $n = 512$ and $m =
128$ with $p = \frac12$ so that $m_d = 64$ and $m_e = 64$.  We then
compute the average MSE over 100,000 iterations for each value of
$\mu$ and for both algorithms.  We compare this to a nonadaptive
procedure which uses a (normalized) i.i.d.\ Rademacher matrix followed
by orthogonal matching pursuit (OMP).  Note that in the worst case the
MSE of the adaptive algorithms is comparable to the MSE obtained by
the nonadaptive algorithm and exceeds the lower bound in
Theorem~\ref{thm:second-minmax} by only a small constant factor.
However, when $\mu$ begins to exceed a critical threshold, the MSE
rapidly decays and approaches the optimal value of $\frac{1}{m_e n}$.
Note that when $\mu$ is large we can take $m_e \rightarrow m$ and
hence can actually get arbitrarily close to $\frac{1}{m n}$ in the
asymptotic regime.

\section{Discussion}
\label{sec:discussion}

The contribution of this paper is to show that if one has the freedom
to choose any adaptive sensing strategy and any estimation procedure
no matter how complicated or computationally intractable, we would not
be able to universally improve over a simple nonadaptive strategy that
simply projects the signal onto a lower dimensional space and perform
recovery via $\ell_1$ minimization.  This ``negative'' result should
not conceal the fact that adaptivity may help tremendously if the SNR
is sufficiently large, as illustrated in Section~\ref{sec:numerics}.
Hence, we regard the design and analysis of effective adaptive schemes
as a subject of important future research.  At the methodological
level, it seems important to develop adaptive strategies and
algorithms for support estimation that are as accurate and as robust
as possible. Further, a transition towards practical applications
would need to involve engineering hardware that can effectively
implement this sort of feedback, an issue which poses all kinds of
very concrete challenges. Finally, at the theoretical level, it would
be of interest to analyze the phase transition phenomenon we expect to
occur in simple Bayesian signal models. For instance, a central
question would be how many measurements are required to transition
from a nearly flat posterior to one mostly concentrated on the true
support.

In closing, we note that after the submission of this paper, a variant
of Algorithm~\ref{alg:2} was shown to recover the correct support of a
$1$-sparse vector with high probability provided that the amplitude
$\mu$ of the nonzero entry obeys $\mu \ge C \sqrt{n/m}$ for some
positive numerical constant $C$ \cite{adaptiveISIT,malloynowak}.  This
implies that for $k=1$, the lower bound in
Theorem~\ref{thm:second-minmax} is tight up to constant factors.
Thus, adaptive methods have the potential to remove the $\log(n/k)$
factor required in the nonadaptive setting.

\small

\subsection*{Acknowledgements}
The authors would like to thank the reviewers as well as Rui Castro,
Jarvis Haupt, and Alexander Tsybakov for their insightful feedback.
They are grateful to Xiaodong Li for suggesting an improvement in the
proof of Theorem~\ref{thm:bernoulli-support} and to Adam Bull for
pointing out a technical error.  E.~A-C.~is partially supported by
ONR grant N00014-09-1-0258. E.~C.~is partially supported by NSF via
grant CCF-0963835 and the 2006 Waterman Award, by AFOSR under grant
FA9550-09-1-0643 and by ONR under grant N00014-09-1-0258. M.~D.~is
supported by NSF grant DMS-1004718.

\bibliographystyle{abbrv}
\bibliography{adaptiveCS}

\end{document}